\documentclass[12pt]{amsart}

\usepackage{amsthm, amsmath}
\usepackage{amssymb}
\usepackage[all]{xy}
\usepackage{mathrsfs}
\usepackage[latin1]{inputenc}
\usepackage{graphicx}
\usepackage{xspace}
\usepackage{hyperref}

\numberwithin{equation}{subsection}

\DeclareMathOperator{\Gal}{Gal}

\DeclareMathOperator{\id}{id}
\DeclareMathOperator{\Adm}{Adm}

\theoremstyle{plain}

\newtheorem{theorem}{Theorem}[section]

\newtheorem{lem}[theorem]{Lemma}
\newtheorem{thm}[theorem]{Theorem}
\newtheorem{cor}[theorem]{Corollary}

\newtheorem{prop}[theorem]{Proposition}

\theoremstyle{definition}
\newtheorem{rem}[theorem]{Remark}
\newtheorem{example}[theorem]{Example}

\def\ge{\geqslant}
\def\le{\leqslant}
\def\a{\alpha}
\def\b{\beta}
\def\g{\gamma}
\def\G{\Gamma}

\def\L{\Lambda}
\def\e{\epsilon}

\def\io{\iota}
\def\o{\omega}

\def\s{\sigma}
\def\t{\tau}
\def\th{\theta}
\def\k{\kappa}
\def\l{\lambda}
\def\z{\zeta}

\def\i{^{-1}}
\def\tSS{\tilde{\mathbb S}}
\def\SS{\mathbb S}

\def\ZZ{\mathbb Z}
\def\NN{\mathbb N}
\def\QQ{\mathbb Q}
\def\RR{\mathbb R}

\def\co{\mathcal O}

\def\cs{\mathcal S}

\def\aa{\mathbf a}
\def\ua{\underline {\mathbf a}}

\def\tt{\tilde \t}

\def\tW{\tilde W}
\def\tw{\tilde w}

\def\Ad{\text{Ad}}

\def\ba{\mathbf a}

\def\lto

\usepackage{amssymb} % symboler fra AMS
\usepackage{latexsym} % LaTeX symboler
\usepackage{amsfonts} % fonte fra AMS
\usepackage{amsmath} % matematik fonte fra AMS
\usepackage{eucal} % 'krlle'-fonte (skriv \mathcal)
\usepackage{bm} % fed skrift i matematik
\usepackage{bbm} % skrift med dobbelt streg til f.eks. talomr�er
\usepackage{graphicx} % mulighed for at inkludere grafik
\usepackage[english]{varioref} % smarte referencer - se nederst
\usepackage[nice]{nicefrac} % p�ere brker
\usepackage[all]{xy}

\newcommand{\A}[1]{\mathbbm{A}^{#1}}

\def\<{\langle}
\def\>{\rangle}
\def\rf{\rfloor}
\def\lf{\lfloor}
\def\rc{\lceil}
\def\lc{\rceil}
\def\ua{\underline{\mathbf{a}}}
\def\ev{\text{av}}
\def\h{\text{h}}
\def\tt{\text{t}}
\def\fs{\mathfrak{S}}
\def\C{\text{Con}}
\def\cyc{\text{cyc}}
\def\slp{\text{sl}}
\def\A{\textbf{av}}
\def\bc{\textbf{c}}

\begin{document}

\title[On the acceptable elements]{On the acceptable elements}

\author[X. He]{Xuhua He}
\address{Department of Mathematics, University of Maryland, College Park, MD 20742, USA and Department of Mathematics, HKUST, Hong Kong}
\email{xuhuahe@math.umd.edu}
\author{Sian Nie}
\address{Institute of Mathematics, Academy of Mathematics and Systems Science, Chinese Academy of Sciences, 100190, Beijing, China}
\email{niesian@amss.ac.cn}

\begin{abstract}
In this paper, we study the set $B(G, \{\mu\})$ of acceptable elements for any $p$-adic group $G$. We show that $B(G, \{\mu\})$ contains a unique maximal element and the maximal element is represented by an element in the admissible subset of the associated Iwahori-Weyl group.
\end{abstract}

\maketitle

\section*{Introduction}

Let $F$ be a finite field extension of $\QQ_p$ and $L$ be the completion of the maximal unramified extension of $F$. Let $G$ be a connected reductive algebraic group over $F$ and $\s$ be the Frobenius morphism. We denote by $B(G)$ the set of $\s$-conjugacy classes of $G(L)$. The set $B(G)$ is classified by Kottwitz in \cite{Ko1} and \cite{Ko2}. This classification generalizes the Dieudonn\'e-Manin classification of isocrystals by their Newton polygons.

Let $\tW$ be the Iwahori-Weyl group of $G$ over $L$. Let $\{\mu\}$ be a geometric conjugacy class of cocharacters of $G$.
Let $\Adm(\{\mu\}) \subseteq \tW$ be the admissible subset of $\tW$ (\cite{R:guide} and \cite{KR1}) and $B(G, \{\mu\})$ be the finite subset of $B(G)$ defined by the group-theoretic version of Mazur's theorem \cite[\S 6]{Ko2}.

The main result of this paper is as follows.

\begin{thm}\label{main}
The set $B(G, \{\mu\})$ contains a unique maximal element and this element is represented by an element in $\Adm(\{\mu\})$.
\end{thm}

For quasi-split groups, this is obvious as the unique maximal element of $B(G, \{\mu\})$ is represented by a translation element. However, it is much more complicated for non quasi-split groups.

This result is an important ingredient in the proof \cite{He00} of the Kottwitz-Rapoport conjecture \cite[Conjecture 3.1]{KR} and \cite[Conjecture 5.2]{R:guide} on the union of affine Deligne-Lusztig varieties. The knowledge of the explicit description of the maximal element of $B(G, \{\mu\})$ is also useful in the study of the $\mu$-ordinary locus, the most general Newton stratum, of Shimura varieties.

\

In fact, Theorem \ref{main} is a statement on the Iwahori-Weyl group $\tW$, the automorphism on $\tW$ induced from the Frobenius morphism $\s$ of $G$ and a cocharacter $\mu$ in $\{\mu\}$. In section 1, we introduce a set $B(\tW, \mu, \s)$ (for any diagram automorphism $\s$ on $\tW$) and reformulate Theorem \ref{main} as a statement on the triple $(\tW, \mu, \s)$. The relation between $B(G, \{\mu\})$ and $B(\tW, \mu, \s)$ is discussed in the Appendix.

This reformulation allows us to relate different diagram automorphisms of the Iwahori-Weyl groups, which plays an essential role in the proof. We show that the set $B(\tW, \mu, \s)$ contains a unique maximal element in section 2. It requires more work to show that this maximal element is represented by an element in the admissible set. We show in section 3 that it suffices to consider the superbasic case and in section 4 that it suffices to consider the irreducible case. In section 5, we prove the statement for the irreducible superbasic case (that is, $\s$ is a diagram automorphism of order $n$ for the Iwahori-Weyl group of $PGL_n$). This completes the proof of Theorem \ref{main}.

\section{Preliminaries}\label{1}

\subsection{} Let $\mathfrak R=(X^*, R, X_*, R^\vee, \Pi)$ be a based reduced root datum, where $R \subseteq X^*$ is the set of roots, $R^\vee \subseteq X_*$ is the set of coroots and $\Pi \subseteq R$ is the set of simple roots. Let $\langle ~ , ~ \rangle : X^* \times X_* \rightarrow \ZZ$ be the natural perfect pairing between $X^*$ and $X_*$. Let $V=X_* \otimes \RR$. For any $\a \in R$, we have a reflection $s_\a$ on $V$ sending $v$ to $v-\langle \a, v\rangle \a^\vee$.

The reflections $s_\a$ generate the {\it finite Weyl group} $W_0$ of $R$. Let $\SS=\{s_\a; \a \in \Pi\}$ be the set of simple reflections. Then $(W_0, \SS)$ is a Coxeter system.

For any $J \subseteq \SS$, let $W_J$ be the subgroup of $W_0$ generated by $J$ and ${}^J W_0=\{w \in W_0; w=\min(W_J w)\}$.

The {\it closed dominant chamber} is the set $$\mathfrak C=\{v \in V; \langle \a, v\rangle \geq 0 \text{ for every } \a \in \Pi\}.$$ Then for any $v \in V$, the set $\{w(v); w \in W_0\}$ contains a unique element in $\mathfrak C$. We denote this element by $\bar v$.

\subsection{} \label{alcove} Set $\tilde R=R \times \ZZ$. For $(\a, k) \in \tilde R$, we have an affine root $\tilde \a=\a+k$ and an affine reflection $s_{\tilde \a}$ on $V$ sending $v$ to $v-(\<\a, v\> - k) \a^\vee$. For any affine root $\tilde \a$, let $H_{\tilde \a}$ be the hyperplane in $V$ fixed by the reflection $s_{\tilde \a}$. Set \begin{gather*}
W_a=\ZZ R^\vee \rtimes W_0=\{t^\l w; \l \in \ZZ R^\vee, w \in W_0\}, \\
\tW=X_* \rtimes W_0=\{t^\l w; \l \in X_*, w \in W_0\}.
\end{gather*}
We call $W_a$ the {\it affine Weyl group} and $\tW$ the {\it extended affine Weyl group}. Let $\text{Aff}(V)$ be the group of affine transformations on $V$. We realize both $W_a$ and $\tW$ as subgroups of $\text{Aff}(V)$, where $t^\l$ acts by translation $v \mapsto v+\l$ on $V$. We may identify $W_a$ with the subgroup of $\text{Aff}(V)$ generated by the affine reflections.

Let $R^+ \subseteq R$ be the set of positive roots determined by $\Pi$. The {\it base alcove} is the set $$\aa=\{v \in V; 0 < \< \a, v\> < 1 \text{ for every } \a \in R^+\}.$$

The set of positive affine roots is $\{\a+k; \a \in R^+, k \geq 1\} \cup \{-\a+k; \a \in R^+, k \geq 0\}$. The affine simple roots are $-\a$ for $\a \in \Pi$ and $\b+1$, where $\b$ runs over maximal positive roots in $R$. Note that the positive roots in $R$ are {\it not} positive as affine roots.

The hyperplanes $H_{\tilde \a}$, for the affine simple roots $\tilde \a$, are exactly the walls of the base alcove $\aa$.
Let $\tilde \SS$ be the set of $s_{\tilde \a}$, where $\tilde \a$ runs over affine simple roots. Then $\SS \subseteq \tilde \SS$ and $(W_a, \tilde \SS)$ is a Coxeter group.

Let $\Omega$ be the isotropy group in $\tW$ of the base alcove $\aa$. Then $\tW=W_a \rtimes \Omega$. We extend the Bruhat order on $W_a$ to $\tW$ as follows: for $w, w' \in W_a$ and $\t, \t' \in \Omega$, we say that $w \t \le w' \t'$ if $\t=\t'$ and $w \le w'$ (with respect to the Bruhat order on the Coxeter group $W_a$). We put $\ell(w \t)=\ell(w)$.

\subsection{}\label{1.3} Let $\s \in \text{Aff}(V)$ be an automorphism of finite order such that $\s(\aa)=\aa$ and the conjugation action of $\s$ stabilizes $\tW$. Then $\s$ induces a bijection on the set of walls of $\aa$ and hence a bijection on $\tilde \SS$. Let $\varsigma \in GL(V)$ denote the linear part of $\s$ with respect to the decomposition $\text{Aff}(V)=V \rtimes GL(V)$. Then $\s$ acts by conjugation on the translation subgroup of $\text{Aff}(V)$ via $\varsigma$: $$\Ad(\s)(t^\xi)=t^{\varsigma(\xi)} \quad \text{ for } \xi \in V.$$ Since we assume that $\s$ normalizes $\tW$, it follows that $\varsigma$ stabilizes $X_*$ and normalizes $W_0$. %As the subgroup of translations is a normal subgroup of $\text{Aff}(V)$, for any $\l \in X_*$, there exists $\l' \in X_*$ with $\Ad(\s)(t^\l)=t^{\l'}$. In other words, $\s$ induces a bijection on $X_*$. This bijection extends linearly to a bijection on $V$, which we denote by $\varsigma$. Then by definition $$\Ad(\s)(t^\l)=t^{\varsigma(\l)}.$$ Note that if $\s(0)=0$, then the two actions of $\s$ and $\varsigma$ on $V$ coincide.

The $\s$-conjugation action on $\tW$ is defined by $w \cdot_\s w'=w w' \Ad(\s)(w) \i$. We have two invariants on the $\s$-conjugacy classes.

Since $\s(\aa)=\aa$, the conjugation action of $\s$ stabilizes $\Omega$. Let $\Omega_\s$ be the set of $\s$-coinvariants on $\Omega$. The Kottwitz map $\k_{\tW, \s}: \tW \to \Omega_\s$ is obtained by composing the natural projection map $\tW \to \tW/W_a \cong \Omega$ with the projection map $\Omega \to \Omega_\s$. It is constant on each $\s$-conjugacy class of $\tW$. This gives one invariant.

Another invariant is given by the Newton map.

For any $w \in \tW$, we consider the element $w \s \in \text{Aff}(V)$. There exists $n \in \NN$ such that $(w \s)^n=t^\xi$ for some $\xi \in X_*$. Let $\nu_{w, \s}=\xi/n$ and $\bar \nu_{w, \s}$ be the unique dominant element in the $W_0$-orbit of $\nu_{w, \s}$. We call $\bar \nu_{w, \s}$ the {\it Newton point} of $w$ (with respect to the $\s$-conjugation action). It is known that $\nu_{w, \s}$ is independent of the choice of $n$ and $\bar \nu_{w, \s}=\bar \nu_{w', \s}$ if $w$ and $w'$ are $\s$-conjugate. Moreover, $w \s t^\xi=t^\xi w \s \in \text{Aff}(V)$. Hence $w \Ad(\s)(t^{\xi})=t^\xi w$. Therefore $\varsigma(\xi)$ and $\xi$ are in the same $W_0$-orbit and \[\tag{a}\bar \nu_{w, \s}=\overline{\varsigma(\nu_{w, \s})}.\]

\subsection{}\label{1.4} Let $\mu$ be a dominant cocharacter, i.e., $\mu \in X_* \cap \mathfrak C$. The $\mu$-admissible set is defined as $$\Adm(\mu)=\{w \in \tW; w \le t^{x(\mu)} \text{ for some }x \in W_0\}.$$

The partial order on $\mathfrak C$ is defined as follows. Let $v, v' \in \mathfrak C$. We say that $v \leq v'$ if $v'-v \in \sum_{\a \in \Pi} \RR_{\geq 0} \a^\vee$.

Let $N$ be the order of $\s$. For $\mu \in X_*$, we define \begin{gather*} \mu^\diamondsuit_\s =\frac{1}{N} \sum_{i=0}^{N-1} \overline{\varsigma^i(\mu)} \in \mathfrak C \\ \mu^\clubsuit_\s=\nu_{t^\mu, \s}=\frac{1}{N} \sum_{i=0}^{N-1} \varsigma^i(\mu). \end{gather*} If $\s(0)=0$, then $\mu^\diamondsuit_\s=\mu^\clubsuit_\s$. Set $$B(\tW, \mu, \s)=\{\bar \nu_{w, \s}; w \in \tW, \k_{\tW, \s}(w)=\k_{\tW, \s}(t^\mu), \bar \nu_{w, \s} \leq \mu^\diamondsuit_\s \}.$$ The elements in $B(\tW, \mu, \s)$ are called the {\it acceptable elements} for $\mu$.

\

The main result of this paper is as follows.

\begin{thm}\label{main'}

(1) The set $B(\tW, \mu, \s)$ contains a unique maximal element $\nu$ (with respect to the partial order $\leq$ on $\mathfrak C$).

(2) There exists an element $w \in \Adm(\mu)$ with $\bar \nu_{w, \s}=\nu$.
\end{thm}

The relation between the sets $B(\tW, \mu, \s)$ and $B(G, \{\mu\})$ will be discussed in the Appendix A.

\section{The maximal element in $B(\tW, \mu, \s)$}

\subsection{}\label{adj} Let $\mathfrak R_{ad}=(\ZZ R, R, X_*^{ad}, R^\vee, \Pi)$ be the root datum of the adjoint group of the reductive group with root datum $\mathfrak R$. Here $X_*^{ad}$ is the dual lattice of $\ZZ R$. The perfect pairing $\langle ~ , ~ \rangle : X^* \times X_* \rightarrow \ZZ$ induces a natural map $\pi: X_* \otimes \RR \to X_*^{ad} \otimes \RR$. Set $\tW_{ad}=X_*^{ad} \rtimes W_0$. Then the map $\pi$ induces a natural map $\pi: \tW \to \tW_{ad}$.

\begin{lem}\label{0-1-1}
Let $v \in X_*^{ad} \otimes \RR$ and $\hat{v}$, $\hat{v}'$ be two lifts of $v$ under $\pi$. Then $\pi \circ \s(\hat{v})=\pi \circ \s(\hat{v}')$.
\end{lem}

\begin{proof}
Note that $\s(\hat{v})=\s(\hat{v}')+\s(\hat{v}-\hat{v}')-\s(0)$. Since $\varsigma$ normalizes $W_0$, it gives a permutation of the hyperplanes $H_\a=\{v \in V; \<\a, v\>=0\}$ for $\a \in R$. As $\hat{v}-\hat{v}'$ lies in the intersection of these hyperplanes, we have $\s(\hat{v}-\hat{v}')-\s(0)=\varsigma(\hat{v}-\hat{v}')$ is still in this intersection. In other words, $\pi(\s(\hat{v}-\hat{v}')-\s(0))=0$. Thus $\pi \circ \s(\hat{v})=\pi \circ \s(\hat{v}')$.
\end{proof}

Now we define a map $\s_{ad}: X_*^{ad} \otimes \RR \to X_*^{ad} \otimes \RR$ by $v \mapsto \pi \circ \s(\hat{v})$, where $\hat{v} \in V$ is a lift of $v \in X_*^{ad} \otimes \RR$ under $\pi$.
By Lemma \ref{0-1-1}, $\s_{ad}$ is well defined. The affine transformation $\s_{ad}$ on $X_*^{ad} \otimes \RR$ induces a conjugation action on $\tW_{ad}$.

It is easy to see that $\pi(\nu_{w, \s})=\nu_{\pi(w), \s_{ad}}$ for $w \in \tW$ and $\pi$ induces a bijection of posets from $B(\tW, \mu, \s)$ to $B(\tW_{ad}, \pi(\mu), \s_{ad})$. The map $\pi$ also induces a bijection of posets from $\Adm(\mu)$ to $\Adm(\pi(\mu))$. Thus Theorem \ref{main'} holds for $B(\tW, \mu, \s)$ if and only if it holds for $B(\tW_{ad}, \pi(\mu), \s_{ad})$.

\begin{lem}\label{trans}
If $\mathfrak R=\mathfrak R_{ad}$, then $\Omega$ acts simply transitively on the set of special vertices of $\aa$.
\end{lem}

\begin{rem}
This Lemma is known to experts. We include a proof for the reader's convenience.
\end{rem}

\begin{proof}
First, the action of $\Omega$ on $V$ stabilizes the set of special vertices of $\aa$.

Let $v \in V$ be a special vertex of $\aa$. By \cite[VI \S 2 Prop. 3]{Bo}, $v \in X_*$. Since $W_0$ acts simply transitively on the set of chambers, there exists a unique $x \in W_0$ such that $x t^{-v}$ sends $\aa$ to an alcove $\aa'$ in the dominant chamber. Since $x t^{-v}(v)=0$, we deduce that $0$ lies in the closure of $\aa'$.

Note that $\aa$ is the unique alcove in the dominant chamber whose closure contains $0$. Hence $\aa'=\aa$ and $x t^{-v} \in \Omega$. In other words, there exists an element in $\Omega$ sending $v$ to $0$. So $\Omega$ acts transitively on the set of special vertices of $\aa$.

Let $w \in \tW$. If $w$ preserves $0$, then $w \in W_0$. Note that $\Omega \cap W_0=\{1\}$. Thus the action of $\Omega$ on the set of special vertices is simply transitive.
\end{proof}

\subsection{}\label{adjoint}
In the rest of this section, we assume that $\mathfrak R=\mathfrak R_{ad}$.

By Lemma \ref{trans}, there exists $\t \in \Omega$ such that $\s_0:=\t \i \s$ preserves $0 \in V$. In the rest of the paper, unless otherwise stated, we denote by $\l$ the dominant cocharacter with $\t \in t^\l W_0$.

Notice that $\s_0$ is a linear action on $V$ and the conjugation action of $\s_0$ stabilizes  the subset $\SS$ of $\tSS$. For simplicity, in the rest of the paper, we will say $\s_0$-orbits in $\SS$ instead of $\Ad(\s_0)$-orbits in $\SS$. %Then $\s_0$ stabilizes $\SS$ inside $\tSS$ and is a diagram automorphism of $W_0$. In particular, $\s_0$ is a linear action on $V$. %Let $\varsigma_0$ be the linear action on $V$ associated to $\s_0$ (see \S \ref{1.3}).

By definition, $\nu_{w, \s}=\nu_{w \t, \s_0}$ for all $w \in \tW$.  By \ref{1.3} (a), $$\bar \nu_{w, \s}=\bar \nu_{w \t, \s_0} \in \mathfrak C^{\s_0}:=\{v \in \mathfrak C; \s_0(v)=v\}.$$

\begin{lem}\label{suites} If $\xi \in X_* \cap \mathfrak C$, then $\xi^\diamondsuit_\s=\xi^\clubsuit_{\s_0}=\nu_{t^\xi, \s_0}$.
\end{lem}

\begin{proof}
Note that $\s_0=\t \i \s$ and is linear. We have $\s_0=x \varsigma$ for some $x \in W_0$. Since $\varsigma$ normalizes $W_0$, we have $\s_0^i \in W_0 \xi^i$ for any $i \in \ZZ$. Therefore $\overline{\varsigma^i(\zeta)}=\overline{\s_0^i(\zeta)}$. Since $\s_0$ stabilizes $\mathfrak C$, $\xi^\diamondsuit_\s=\xi^\diamondsuit_{\s_0}=\xi^\clubsuit_{\s_0}$.
\end{proof}

\subsection{} For any $i \in \SS$, let $\o^\vee_i \in V$ be the corresponding fundamental coweight and $\a_i^\vee \in V$ be the corresponding simple coroot. We denote by $\o_i, \a_i \in V^*$ the corresponding fundamental weight and corresponding simple root, respectively.

For each $\s_0$-orbit $c$ of $\SS$, we set $$\o_c=\sum_{i \in c} \o_i.$$ For any $v \in \mathfrak C$, we set
\begin{gather*} J(v)=\{s \in \SS; s(v)=v\}, \\ I(v)=\SS \smallsetminus J(v).
\end{gather*} If $v=\s_0(v)$, then both $J(v)$ and $I(v)$ are $\s_0$-stable.

The following lemma is essentially contained in \cite[\S 7.1]{Chai}. Due to its importance, we provide a proof for completeness.
\begin{lem}\label{Newton}
Let $v \in \mathfrak C^{\s_0}$. Then $v =\nu_{w, \s}$ for some $w \in t^\mu W_a$ if and only if $\<\o_c, \mu^\clubsuit_{\s_0} + \l^\clubsuit_{\s_0}-v\> \in\ZZ$ for any $\s_0$-orbit $c$ of $I(v)$.
\end{lem}
\begin{proof}

Suppose $\nu_{w, \s}=v$. We have $w \t=t^\g x$ for some $\g \in X_*$ and $x \in W_0$. By definition, $(w\s)^n=t^{n v}$ for some $n \in \NN$. Since $t^{n v}$ commutes with $w\s=t^\g x\s_0$, we have $x\s_0(v)=x(v)=v$. Thus $x \in W_{J(v)}$. Let $N_0$ be the order of the finite subgroup of $\text{Aff}(V)$ generated by $W_0$ and $\s_0$. Then
\begin{align*} \nu_{w, \s}=\nu_{w \t, \s_0} &=\frac{1}{N_0} \sum_{k=0}^{N_0-1} (x \s_0)^k(\g) \\ &=\frac{1}{N_0} \sum_{k=0}^{N_0-1} (x \Ad(\s_0)(x) \cdots \Ad(\s_0)^{k-1}(x)) \s_0^k(\g) \\ & \in \frac{1}{N_0} \sum_{k=0}^{N_0-1} \s_0^k(\g)+ \sum_{j \in J(v)} \QQ \a_j^\vee \\ &=\g^\clubsuit_{\s_0} + \sum_{j \in J(v)} \QQ \a_j^\vee.\end{align*}

If $w \in t^\mu W_a$, then $w \t \in t^{\mu+\l} W_a$ and $\mu+\l-\g \in \ZZ R^\vee$. Hence $\<\o_c, \mu^\clubsuit_{\s_0}+\l^\clubsuit_{\s_0}-v\>=\<\o_c, \mu^\clubsuit_{\s_0}+\l^\clubsuit_{\s_0}-\g^\clubsuit_{\s_0}\>=\<\o_c, \mu+\l-\g\> \in \ZZ$.

On the other hand, suppose $a_c=\<\o_c, \mu^\clubsuit_{\s_0} + \l^\clubsuit_{\s_0}-v\> \in \ZZ$ for each $\s_0$-orbit $c$ of $I(v)$. We also set $a_c=0$ if $c \nsubseteq I(v)$. We construct an element $w \in t^\mu W_a$ such that $\nu_{w,\s}=v$.

For each $\s_0$-orbit of $J(v)$, we choose a representative. Let $y$ be the product of these representatives (in some order). Then $y$ is a $\s_0$-twisted Coxeter element of $W_{J(v)}$ in the sense of \cite[7.3]{Spr}. For each $\s_0$-orbit $c$ of $I(v)$, we choose a representative $i_c$. Let $\a^\vee_{i_c}$ be the corresponding simple coroot. Set $\b=\mu+\l-\sum_c a_c \a^\vee_{i_c}$ and $w=t^{\b} y \t \i \in t^\mu W_a$. Write $\b=h+r$ with $r \in \sum_{j \in J(v)} \QQ \a_j^\vee$ and $h \in \sum_{i \in I(v)} \QQ \o_i^\vee$. Then
\begin{align*} \nu_{w, \s}=\nu_{w \t, \s_0} &=\frac{1}{N_0} \sum_{k=0}^{N_0-1} (y \s_0)^k(\b) \\ &=h^\clubsuit_{\s_0} + \frac{1}{N_0} \sum_{k=0}^{N_0-1}(y \s_0)^k(r) \\ &=h^\clubsuit_{\s_0}=\mu^\clubsuit_{\s_0}+\l^\clubsuit_{\s_0}-\sum_c a_c (\a^\vee_{i_c})^\clubsuit_{\s_0} - r^\clubsuit_{\s_0},\end{align*} where the fourth equality follows from \cite[Lemma 7.4]{Spr}.

Hence for any $\s_0$-orbit $c$ of $I(v)$ and any $j \in J(v)$, we have $$\<\o_{c}, \mu^\clubsuit_{\s_0}+\l^\clubsuit_{\s_0}-\nu_{w, \s}\>=\<\o_{c}, \sum_{c'} a_{c'} (\a^\vee_{i_{c'}})^\clubsuit_{\s_0}\>=a_{c}$$ and $$\<\a_j, \mu^\clubsuit_{\s_0}+\l^\clubsuit_{\s_0}-\nu_{w,\s}\>=\<\a_j, \mu^\clubsuit_{\s_0}+\l^\clubsuit_{\s_0}\>=\<\a_j, \mu^\clubsuit_{\s_0}+\l^\clubsuit_{\s_0}-v\>,$$ which means $\nu_{w,\s}=v$ as desired.
\end{proof}

\begin{cor}
$\mu^\diamondsuit_{\s}=\mu^\clubsuit_{\s_0} \in B(\tW, \mu, \s)$ if and only if $\<\o_c, \l^\clubsuit_{\s_0}\> \in \ZZ$ for any $\s_0$-orbit $c$ of $I(\mu^\diamondsuit_{\s})$. In this case, $\mu^\diamondsuit_{\s}$ is a priori the unique maximal element of $B(\tW, \mu, \s)$.
\end{cor}

\subsection{} We follow \cite[\S 6]{Chai}. For any $\s_0$-stable subset $B$ of $\mathfrak C$, we define
$$C_{\geq B}=\{v \in \mathfrak C^{\s_0}; v \geq b, \forall b \in B\}.$$ We say $B$ is {\it reduced} if $C_{\geq B} \subsetneq C_{\geq B'}$ for any $\s_0$-stable proper subset $B' \subsetneq B$.

For any $i \in \SS$, let $$pr_{(i)}: V=\RR \o^\vee_i \oplus \sum_{j \neq i} \RR \a^\vee_j \to \RR \o^\vee_i$$ be the projection map.

Now we prove part (1) of Theorem \ref{main'}.

\subsection{Proof of Theorem \ref{main'} (1)} \label{p1}
By \S\ref{adj}, it suffices to consider the case where $\mathfrak R=\mathfrak R_{ad}$.

For any $i \in \SS$, let $c$ denote the $\s_0$-orbit of $i$, and define $e_i \in \QQ \o_i^\vee$ by $$\<\o_i, e_i\>=\frac{1}{\# c}\max(\{t \in \<\o_c, \mu^\clubsuit_{\s_0}+\l^\clubsuit_{\s_0}\> + \ZZ; t \leq \<\o_c, \mu^\clubsuit_{\s_0}\> \} \cup \{0\}).$$ Let $E_0=\{e_i; i \in \SS\}$. It is easy to prove by induction on the number of $\s_0$-orbits on $E_0$ that there exists a $\s_0$-stable subset $E$ of $E_0$ which is reduced and satisfies $C_{\geq E}=C_{\geq E_0}$. Let $I(E)=\{i \in \SS; e_i \in E\}$. By \cite[Theorem 6.5]{Chai}\footnote{In fact, we use here a ``$\s_0$-fixed'' version of \cite[Theorem 6.5]{Chai}, which can be proved in the same way as in loc.cit.}, there exists an element $\nu \in C_{\geq E}$ defined by $I(\nu)=I(E)$ and $\<\o_j, \nu\>=\<\o_j, e_j\>$ for $j \in I(E)$, which satisfies $C_{\geq \nu}=C_{\geq E}=C_{\geq E_0}$. Since $\mu^\diamondsuit_\s=\mu^\clubsuit_{\s_0} \in C_{\geq E_0}$, we have $\nu \leq \mu^\diamondsuit_\s$. By Lemma \ref{Newton}, $\nu \in B(\tW, \mu, \s)$.

Since $\nu \in C_{\geq E_0}$, $\nu \geq e_i$ for any $i \in \SS$. Therefore, for any $\s_0$-orbit $c$ of $\SS$, we have \begin{align*}\tag{a} \<\o_c, \mu^\clubsuit_{\s_0}\> \geq \<\o_c, \nu\> \geq \sum_{j \in c} \<\o_j, e_j\> \geq \<\o_c, \mu^\clubsuit_{\s_0}+\l^\clubsuit_{\s_0}\>-\rc \<\o_c, \l^\clubsuit_{\s_0}\> \lc, \end{align*} where the last inequality follows from our definition of $e_j$ for $j \in \SS$.

Let $\nu' \in B(\tW,\mu, \s)$. Set $E(\nu')=\{pr_{(j)}(\nu'); j \in I(\nu')\}$. By Lemma \ref{Newton} and the inequality $\nu' \leq \mu^\diamondsuit_\s=\mu^\clubsuit_{\s_0}$, we have, for any $\s_0$-orbit $c$ of $I(\nu')$ and $j \in c$, that $$\#{c} \cdot \<\o_j, pr_{(j)}(\nu')\>=\#{c} \cdot \<\o_j, \nu'\>=\<\o_c, \nu'\> \in \<\o_c, \mu^\clubsuit_{\s_0} +\l^\clubsuit_{\s_0}\> + \ZZ$$ and $$\# c \cdot \<\o_j, pr_{(j)}(\nu')\> \leq \# c \cdot \<\o_j, \mu^\clubsuit_{\s_0}\>=\<\o_c, \mu^\clubsuit_{\s_0}\>.$$ So $\<\o_j, pr_{(j)}(\nu')\> \leq \<\o_j, e_j\>$, that is, $pr_{(j)}(\nu') \leq e_j \leq \nu$ for $j \in I(\nu')$. By \cite[Lemma 6.2 (i)]{Chai}, we deduce that $\nu' \leq \nu$. Therefore $\nu$ is the unique maximal element of $B(\tW,\mu, \s)$.

\section{Reduction to the superbasic case}

\subsection{}\label{4.1} In the rest of the paper we prove Theorem \ref{main'} (2), beginning in this section with a reduction step to the superbasic case.

For any element $w \s^i$ with $w \in \tW$ and $i \in \ZZ$, we put $\ell(w \s^i)=\ell(w)$. This is well-defined since $\s(\aa)=\aa$.

Let $\e=w \s^i$ with $\ell(\e)=0$. Then the conjugation action of $\e$ on $\tW$ sends simple reflections to simple reflections. We say that $\e$ is {\it superbasic} (for $\tW$) if each $\Ad(\e)$-orbit on $\tSS$ is a union of connected components of the affine Dynkin diagram of $\tW$. By \cite[3.5]{HN}, $\e$ is superbasic if and only if $W_a=W_1^{m_1} \times \cdots \times W_l^{m_l}$, where each $W_i$ is an extended affine Weyl group of type $\tilde A_{n_i-1}$ and $\e$ gives an order $n_i m_i$ permutation on the set of simple reflections of $W_i^{m_i}$.

\subsection{} Let $J \subseteq \SS$. Let $\tW_J=X_* \rtimes W_J$ be the corresponding parabolic subgroup of $\tW$. This is the extended affine Weyl group associated to the root datum $\mathfrak R_J=(X^*, R_J, X_*, R^\vee_J, \Pi_J)$, where $\Pi_J$ is the subset of simple roots corresponding to $J$ and $R_J \subseteq R$ is the set of roots spanned by $\Pi_J$. Set $R^+_J=R_J \cap R^+$. Let $\aa_J=\{v \in V; 0 < \< \a, v\> < 1 \text{ for every } \a \in R^+_J\}$ be the base alcove associated to $\tW_J$.

The set of positive affine roots $\tilde R_J$ for $\tW_J$ is $\{\a+k; \a \in R^+_J, k \geq 1\} \cup \{-\a+k; \a \in R^+_J, k \geq 0\}$. The affine simple roots for $\tW_J$ are $-\a$ for $\a \in \Pi_J$ and $\b+1$, where $\b$ runs over maximal positive roots in $R_J^+$. Note that the positive roots in $R_J$ are {\it not} positive as affine roots in $\tilde R_J$.

We denote by $\le_J$ and $\ell_J$ the Bruhat order and length function on $\tW_J$. Although $\tW_J$ is a subgroup of $\tW$, $\le_J$ and $\ell_J$ can be quite different from the restrictions of $\le$ and $\ell$ to $\tW_J$.

\subsection{}\label{4.2}

In the rest of this section, we assume that $\mathfrak R=\mathfrak R_{ad}$. We take $\t \in \Omega$ and $\s_0=\t \i \s$ as in \S\ref{adjoint}. Recall that $\l$ is the dominant cocharacter with $\t \in t^\l W_0$.

We will associate to $\s$ a superbasic element for a parabolic subgroup of $\tW$ and reduce Theorem \ref{main'} (2) to the superbasic case.

We follow the approach in \cite[\S 5]{HN2}.

Let $V^\s$ be the fixed point set of $\s$. Since $\s$ is an affine transformation on $V$ of finite order, $V^\s$ is a nonempty affine subspace. Set $V'=\{v-e; v \in V^\s\}$, where $e$ is an arbitrary point of $V^\s$. Then $V'$ is the (linear) subspace of $V$ parallel to $V^\s$. We choose a generic point $v_0$ of $V'$, i.e., for any root $\a \in R$, $\<\a, v_0\>=0$ implies that $\<\a, v'\>=0$ for all $v' \in V'$. We set $I=I(\bar v_0)$, $J=J(\bar v_0)$ and $\s^J=z \s z \i \in \tW \s$, where $z$ is the unique element in ${}^J W_0$ with $\bar v_0=z(v_0)$.

\begin{lem} \label{prep}
(1) The set $J$ is stable under $\s_0$-conjugation.

(2) $z(\l)^\clubsuit_{\s_0} \in \QQ R_J^\vee$.

(3) The element $\s^J$ is a superbasic element for $\tW_J$.
\end{lem}

\begin{proof}
(1) By definition, $\s(0)=\l$. Hence $\s(v_0)=v_0+\l$ and $\s^J(\bar v_0)=\bar v_0+z(\l)$. Write $\s^J$ as $\s^J=t^{z(\l)}u \s_0$ for some $u \in W_0$. Then $u\s_0(\bar v_0)=\bar v_0$. Therefore $\s_0(\bar v_0)=u\i(\bar v_0)$ is the unique dominant element in the $W_0$-orbit of $v_0$. Hence $\bar v_0=\s_0(\bar v_0)=u \i(\bar v_0)$. Therefore $u \in W_J$ and $\Ad(\s_0)(J)=J$.

(2) Since $\s^J$ is of finite order, we have $(\s^J)^m=1$ for some $m \in \NN$. On the other hand, using the expression $\s^J=t^{z(\l)}u \s_0$, one computes that $(\s^J)^m=t^{\sum_{k=0}^{m-1} (u\s_0)^k(z(\l))}=1$. So $\sum_{k=0}^{m-1} (u\s_0)^k(z(\l))=0$. Since $u \in W_J$ and $\s_0(R_J^\vee)=R_J^\vee$, we have $(u\s_0)^k(z(\l))-\s_0^k(z(\l)) \in \ZZ R_J^\vee$ for $k \in \ZZ$. Thus $(z(\l))^\clubsuit_{\s_0} \in \frac{1}{m}\sum_{k=0}^{m-1} (u\s_0)^k(z(\l)) + \QQ R_J^\vee=\QQ R_J^\vee$.

(3) Since $\s_0$ stabilizes $\aa_J$, the length function $\ell_J$ on $\tW_J$ extends to the subgroup of $\text{Aff}(V)$ generated by $\tW_J$ and $\s_0$ via the usual rule $\ell_J(w \s_0^i)=\ell_J(w)$ for $w \in \tW_J$ and $i \in \ZZ$. Since $z\i(R_J^+) \subseteq R^+$, we have $z(\aa) \subseteq \aa_J$. In other words, $\aa_J$ is the unique alcove associated to $\tilde W_J$ that contains $z(\aa)$. Since $\s^J(z(\aa))=z(\aa)$, $\s^J(\aa_J)$ is also the unique alcove associated to $\tilde W_J$ that contains $z(\aa)$. Therefore $\s^J(\aa_J)=\aa_J$.

Since $v_0$ is generic in $V'$, $\bar v_0=z(v_0)$ is generic in $z(V')$. So each point of $z(V')$ is fixed by $W_J$. Therefore, for any $\tilde \a \in \tilde R_J$, either $V^{\s^J} \cap H_{\tilde \a} = \emptyset$ or $V^{\s^J} \subseteq H_{\tilde \a}$, where $V^{\s^J}=z(V^\s)$ is the fixed-point set of $\s^J$ on $V$. Since $\s^J(\aa_J)=\aa_J$ and $\s^J$ is of finite order, $\aa_J$ contains a fixed point of $\s^J$. So $V^{\s^J} \nsubseteq H_{\tilde \a}$ and hence $V^{\s^J} \cap H_{\tilde \a}=\emptyset$. By \cite[Proposition 3.5]{HN} (for $\tW_G:=\tW_J, J_\co:=J$ and $y=1$), $\s^J$ is superbasic for $\tW_J$.
%Therefore there is no nonempty subset of $J$ that is stable under $\s^J$. Hence each orbit of $\s^J$ on the set of simple reflections of $\tW_J$ is a union of connected components of the affine Dynkin diagram of $\tW_J$. Hence $\s^J$ is superbasic.
\end{proof}

\begin{lem}\label{integer} Let $c$ be a $\s_0$-orbit of $\SS$. Then $\<\o_c, \l^\clubsuit_{\s_0}\> \in \ZZ$ if $c \subseteq I$.
\end{lem}
\begin{proof}
Write $\l=z(\l)+\th$ for some $\th \in \ZZ R^\vee$. We have $$\<\o_c, \l^\clubsuit_{\s_0}\>=\<\o_c, z(\l)^\clubsuit_{\s_0}\> + \<\o_c, \th\> \equiv \<\o_c, z(\l)^\clubsuit_{\s_0}\> \mod \ZZ.$$
By Lemma \ref{prep} (2), $\<\o_c, z(\l)^\clubsuit_{\s_0}\>=0$ if $c \subseteq I$. The proof is finished.
\end{proof}

\begin{prop}\label{image}
The maximal Newton point of $B(\tW, \mu, \s)$ is contained in the natural inclusion $B(\tW_J, \mu, \s^J) \hookrightarrow B(\tW, \mu, \s)$.
\end{prop}
\begin{proof}
For any $j \in J$, we denote by $\o_j^J$ the fundamental weight corresponding to $j$ in the root datum $\mathfrak R_J$. We set $\o_c^J=\sum_{j \in c} \o_j^J$ for any $\s_0$-orbit of $c$ of $J$. Let $\nu$ be the maximal Newton point of $B(\tW, \mu, \s)$.

Let $c$ be a $\s_0$-orbit of $I$. By Lemma \ref{integer}, $\<\o_c, \l^\clubsuit_{\s_0}\> \in \ZZ$. Applying \S \ref{p1} (a), we see that $\<\o_c, \mu^\clubsuit_{\s_0}\>=\<\o_c, \nu\>$ and $\mu^\clubsuit_{\s_0} - \nu \in \QQ R_J^\vee$.

By Lemma \ref{prep} (2), $z(\l)^\clubsuit_{\s_0} \in \QQ R_J^\vee$. Thus $\mu^\clubsuit_{\s_0}+z(\l)^\clubsuit_{\s_0}-\nu \in \QQ R_J^\vee$. Now let $c'$ be a $\s_0$-orbit in $I(\nu) \cap J$. Then
$$\<\o_{c'}^J, \mu^\clubsuit_{\s_0}+z(\l)^\clubsuit_{\s_0}-\nu\> = \<\o_{c'}, \mu^\clubsuit_{\s_0}+z(\l)^\clubsuit_{\s_0}-\nu\>=\<\o_{c'}, \mu^\clubsuit_{\s_0}+\l^\clubsuit_{\s_0}-\nu \> - \<\o_{c'}, \th\>,$$ where $\th=\l-z(\l) \in \ZZ R^\vee$. By Lemma \ref{Newton} $\<\o_{c'}, \mu^\clubsuit_{\s_0}+\l^\clubsuit_{\s_0}-\nu\> \in \ZZ$. Hence $\<\o_{c'}^J, \mu^\clubsuit_{\s_0}+z(\l)^\clubsuit_{\s_0}-\nu\> \in \ZZ$. Again by Lemma \ref{Newton}, we have $\pi_J(\nu) \in B((\tW_J)_{ad}, \pi_J(\mu), (\s^J)_{ad})$, where $\pi_J$ and (resp. $(\s^J)_{ad}$) is defined similarly as $\pi$ (resp. $\s_{ad}$) in \S \ref{0-1-1} with $\mathfrak R$ and $\s$ replaced by $\mathfrak R_J$ and $\s^J$ respectively. Since $\mu^\clubsuit_{\s_0}+z(\l)^\clubsuit_{\s_0}-\nu \in \QQ R_J^\vee$ and $$\pi_J: B(\tW_J, \mu, \s^J) \to B((\tW_J)_{ad}, \pi_J(\mu), (\s^J)_{ad})$$ is a bijection of posets, we deduce that $\nu \in B(\tW_J, \mu, \s^J)$ as desired.
\end{proof}

\begin{lem}
Let $K \subseteq \SS$ and $z \in {}^K W_0$. If $w, w' \in \tW_K$ with $w \le_K w'$ for the Bruhat order of $\tW_K$, then $z \i w z \le z \i w' z$ for the Bruhat order of $\tW$.
\end{lem}

\begin{proof}
By the definition of Bruhat order, there exist positive affine roots $\tilde \a_1, \cdots, \tilde \a_k$ of $\tW_K$ such that $$w \le_K w s_{\tilde \a_1} \le_K w s_{\tilde \a_1} s_{\tilde \a_2} \le_K \cdots \le_K w s_{\tilde \a_1} \cdots s_{\tilde \a_k}=w'.$$

Hence for any $i$, $w s_{\tilde \a_1} \cdots s_{\tilde \a_i}(\tilde \a_{i+1})$ is a positive affine root of $\tW_K$.

Notice that $z \i$ sends positive affine roots of $\tW_K$ to positive affine roots of $\tW$. Set $\tilde \b_i=z \i(\tilde \a_i)$. This is a positive affine root of $\tW$. Moreover, $(z \i w z) s_{\tilde \b_1} \cdots s_{\tilde \b_i}(\tilde \b_{i+1})=z\i w s_{\tilde \a_1} \cdots s_{\tilde \a_i} (\tilde \a_{i+1})$ is a positive affine root of $\tW$. Thus $$z \i w z \le z \i w z s_{\tilde \b_1} \le z \i w z s_{\tilde \b_1} s_{\tilde \b_2} \le \cdots \le z \i w z s_{\tilde \b_1} \cdots s_{\tilde \b_k}=z \i w' z.$$
\end{proof}

\begin{cor}\label{super}
If Theorem \ref{main'} (2) holds for $B(\tW_J, \mu, \s^J)$, then it holds for $B(\tW, \mu, \s)$.
\end{cor}
\begin{proof}
Let $\nu$ be the maximal Newton point of $B(\tW, \mu, \s)$, which is also the maximal Newton point of $B(\tW_J, \mu, \s^J)$ by Proposition \ref{image}. By assumption, there exist $w_1 \in t^\mu (W_a \cap \tW_J)$ and $x_1 \in W_J$ such that $\bar \nu_{w_1, \s^J}^J=\nu$ and $w_1 \le_J t^{x_1(\mu)}$, where $\bar \nu^J_{w_1, \s^J}$ stands for the Newton point of $w_1 \in \tW_J$ defined with respect to the $\s^J$-conjugation action on $\tW_J$. Let $z$ be the element defined in \S \ref{4.2}. Let $w=z\i w_1 z$ and $x=z\i x_1$. Then we have $\bar \nu_{w, \s}=\nu$, $w \in t^\mu W_a$ and $w \le t^{x(\mu)}$ as desired.
\end{proof}

\section{Reduction to the irreducible case}

\subsection{}\label{3.1} In this section, we assume that $\mathfrak R=\mathfrak R_{ad}$  and $\s$ acts transitively on the set of connected components of the affine Dynkin diagram of $W_a$. In other words, $\tW=\tW_1 \times \cdots \times \tW_m$, where $\tW_1 \cong \cdots \cong \tW_m$ are extended affine Weyl groups of adjoint type with connected affine Dynkin diagram and $\Ad(\s)(\tW_1)=\tW_2, \cdots, \Ad(\s)(\tW_m)=\tW_1$. Let $W_i$ be the finite Weyl group associated to $\tilde W_i$.

As in \S \ref{adjoint}, we write $\s$ as $\s=\t \s_0$ with $\t \in \Omega$ and $\Ad(\s_0)(\SS)=\SS$. Write $\mu$ as $\mu=(\mu_1, \cdots, \mu_m)$, where each $\mu_i$ is a dominant cocharacter for $\tW_i$. Let $y=(w_1, \dots, w_m) \in \tW$. Then the $m$-th component of $(y\s)^m \s^{-m} \in \tW$ is $w_m\cdots \Ad(\s^{m-1})(w_1)$. The map $\bar \nu_{(w_1, \dots, w_m), \s} \mapsto \bar \nu_{w_m\cdots \Ad(\s^{m-1})(w_1), \s^m}$ induces a natural surjection from $B(\tW, \mu, \s)$ to $B(\tW_m, \g, \s^m)$, where $\g=\sum_{i=1}^m \s_0^{m-i} (\mu_i)$. Since the elements in $B(\tW, \mu, \s)$ are $\s_0$-invariant, it is in fact a bijection, whose inverse is given by $v \mapsto \frac{1}{m}(\s_0(v), \s_0^2(v), \cdots, v)$. It is easy to see this bijection is a bijection of posets.

\begin{lem}\label{irreducible}
If Theorem \ref{main'} (2) holds for $(\tW_m, \g, \s^m)$, then it holds for $(\tW, \mu, \s)$.
\end{lem}
\begin{proof}
Let $\nu$ be the maximal element in $B(\tW_m, \g, \s^m)$. Then the maximal element in $B(\tW, \mu, \s)$ is $\frac{1}{m}(\s_0(\nu), \cdots, \nu)$. By assumption, there exists $w \in \Adm(\g)$ such that $\bar \nu_{w, \s^m}=\nu$. By definition, there exists $x \in W_m$ such that $w \le t^{x(\g)}$. Since $\ell(t^{x(\g)})=\sum_{i=1}^m \ell(t^{x(\s_0^{m-i}(\mu_i))})$, there exists $w_i \in \tW_m$ for each $i$ such that $w=w_m \cdots w_1$ and $w_i \le t^{x(\s_0^{m-i}(\mu_i))}$ for all $i$. It is easy to see that $\s^{i-m}=\t'_i \s_0^{i-m}$ for some $\t'_i \in \tW$. Hence \begin{align*} \Ad(\s)^{i-m}(w_i) & \le \Ad(\s)^{i-m}(t^{x(\s_0^{m-i}(\mu_i))}) \\ &=\Ad(\t'_i) \Ad(\s_0)^{i-m} (t^{x(\s_0^{m-i}(\mu_i))}) \\ &= t^{x_i(\mu_i)} \end{align*} for some $x_i \in W_i$. Set $y=(\Ad(\s)^{1-m}(w_1), \cdots, w_m) \in \tW$. Then $y \in \Adm(\mu)$. Notice that the $m$-th component of $(y \s)^m \s^{-m}$ is $w_m \cdots w_1=w$. Hence $\bar \nu_{y, \s}=\frac{1}{m}(\s_0(\nu), \cdots, \nu)$.
\end{proof}

\section{The irreducible superbasic case}

\subsection{} In this section, we consider the extended affine Weyl group $\tW=\ZZ^n \rtimes \fs_n$ of type $\tilde A_{n-1}$, where $\fs_n$ is the permutation group of $\{1, 2, \dots, n\}$ which acts on $\ZZ^n \cong \oplus_{i=1}^n \ZZ e_i^\vee$ by $w(e_i^\vee)=e_{w(i)}^\vee$ for $w \in \fs_n$. Let $\{e_i\}_{i=1, \cdots, n}$ be the dual basis. Set $d=\sum_{i=1}^n e_i$ and $d^\vee=\sum_{i=1}^n e_i^\vee$. The simple roots, fundamental weights and fundamental coweights are given by $\a_i=e_i-e_{i+1}$, $\o_{i,n}=-\frac{i}{n}d+\sum_{j=1}^i e_j$ and $\o_{i,n}^\vee=-\frac{i}{n}d^\vee+\sum_{j=1}^i e_j^\vee$ respectively for $i$ with $1 \le i \le n-1$. Then $\tW_{ad}=(\oplus_{i=1}^{n-1} \ZZ \o_{i, n}^\vee) \rtimes \fs_n$, see \S \ref{adj}.

Set $\varpi_{m,n}=\sum_{j=1}^m e_j^\vee$. The map $\pi$ in \S\ref{adj} can be described explicitly as the $\QQ$-linear projection  $\pi: \QQ^n \to \QQ R^\vee \subseteq \QQ^n$ such that $d^\vee \mapsto 0$ and $\varpi_{m,n} \mapsto \o_{m,n}^\vee$ for $m$ with $1 \le m \le n-1$. We also denote the induced (surjective) projection $\tW \to \tW_{ad}$ by $\pi$.

For any positive integer $m<n$, let $\s_{m, n}=t^{\varpi_{m,n}} u_{m,n} \in t^{\varpi_{m,n}} \fs_n$ be the unique length zero element with $u_{m, n} \in \fs_n$. Then any superbasic element in $\tW_{ad}$ is of the form $\pi(\s_{m, n})$ for some positive integer $m<n$ co-prime to $n$. %As we are only interested in the conjugation action here, it suffices to consider only the elements $\s_{m, n}$.

The main purpose of this section is to prove the following result.

\begin{prop}\label{superbasic}
Let $m<n$ be a positive integer co-prime to $n$. Let $\mu \in \oplus_{i=1}^n \ZZ e_i^\vee$ be a dominant cocharacter. Then there exist $\tw \in \tW$ and $x \in \fs_n$ such that $\tw \le t^{x(\mu)}$ and $\pi(\bar \nu_{\tw, \s_{m,n}})=\bar \nu_{\pi(\tw), \pi(\s_{m, n})}$ equals the unique maximal Newton point $\nu$ of $B(\tW_{ad}, \pi(\mu), \pi(\s_{m,n}))$.
\end{prop}

The proof will be given in \S \ref{proof-5.1}.

\subsection{} We first show that Proposition \ref{superbasic} implies Theorem \ref{main'} (2) for any triple $(\tW_1, \mu_1, \s_1)$.

Let $\mathfrak R$ be the root datum of $\tW_1$. By \S\ref{adj}, we may assume $\mathfrak R=\mathfrak R_{ad}$. By Corollary \ref{super}, it suffices to prove Theorem \ref{main'} (2) for $(\tW_2, \mu_2, \s_2)$, where $\s_2$ is a superbasic element in $\tW_2$. By \S\ref{adj} again, we may assume that the root datum of $\tW_2$ is adjoint. By \S \ref{4.1}, we may assume furthermore that $\tW_2=\tW_3^m$, where $\tW_3$ is the extended affine Weyl group of an adjoint root datum of type $A$ and $\s_2$ acts transitively on the set of affine simple reflections of $\tW_2$. By Lemma \ref{irreducible}, it suffices to prove Theorem \ref{main'} (2) for $(\tW_3, \mu_3, \s_3)$, where $\s_3=\s_2^m$ is a superbasic element in $\tW_3$. This case follows from Proposition \ref{superbasic}.

\subsection{} \label{5.3} We recall the definition of $\ua$-sequence and $\chi_{m, n}$ in \cite[\S 3 \& \S 5]{He10}.

For $i, j \in \ZZ$, we set $[i, j]=\{k \in \ZZ; i \le k \le j\}$.

Let $r \in \NN$ and $\chi \in \ZZ^r$. For each $j \in [1, r]$ we define $\ua_\chi^j: \ZZ_{\geq 0} \to \ZZ$ by $\ua_\chi^j(k)=\chi(j-k)$. Here we identify $l$ with $l+r$ for $l \in \ZZ$. We say $i \geq_\chi j$ if $\ua_\chi^i \geq \ua_\chi^j$ in the sense of lexicographic order. If $\ge_\chi$ is a linear order, we define $\e_\chi \in \fs_r$ such that $\e_\chi(i) < \e_\chi(j)$ if and only if $i >_\chi j$.

Let $s \leq r$ be two nonnegative integers which are co-prime. Define $\chi_{s, r} \in \ZZ^r$ by $\chi_{s, r}(i)=\lf \frac{is}{r} \rf-\lf (i-1) \frac{s}{r} \rf $ for $i \in [1, r]$. Set $\e_{s, r}=\e_{\chi_{s, r}}$, which is well defined since $s$ and $r$ are co-prime.

The following explicit description of $\geq_{\chi_{m,n}}$ and its application to the proof of Lemma \ref{equa} below are kindly suggested by one of the anonymous referees. First we identify $[1, n]$ with $\ZZ / n \ZZ$ in the natural way. Then one checks by definition that \begin{align*} \tag{a} \chi_{m,n}(i)= \begin{cases} & 1,  \text{ if } mi \in [0, m-1] \subseteq \ZZ / n\ZZ \\ & 0, \text{ otherwise. } \end{cases} \end{align*} Since $m$ is co-prime to $n$, $m$ is invertible in $\ZZ / n \ZZ$. We claim that \begin{align*} \tag{b} n=m\i \cdot 0 >_{\chi_{m,n}} m\i \cdot 1 >_{\chi_{m,n}} \cdots >_{\chi_{m,n}} m\i \cdot (n-1).\end{align*} Indeed, if $m\i \cdot j_0 <_{\chi_{m,n}} m\i \cdot (j_0+1)$ for some $j_0 \in [0, n-2]$, there exists $0 \leq k \leq n-1$ such that $\chi_{m,n}(m\i \cdot j_0-i) = \chi_{m,n}(m\i \cdot (j_0+1)-i)$ for $0 \leq i \leq k-1$ and $\chi_{m,n}(m\i \cdot j_0-k) < \chi_{m,n}(m\i \cdot (j_0+1)-k)$. In particular, we have $j_0+1-km=0 \in \ZZ / n\ZZ$ by (a). Thus $k \geq 1$ and $$\chi_{m,n}(m\i \cdot j_0-k+1)=1 > 0=\chi_{m,n}(m\i \cdot (j_0+1)-k+1),$$ which is a contradiction. So (b) is proved. Now it is easy to see that $\e_{m,n}$ is the permutation on $[1, n]=\ZZ / n\ZZ$ given by \begin{align*}\tag{c} \e_{m,n}(i)=mi+1 \text{ for } i \in [1, n]. \end{align*}

\subsection{}\label{polygon} %Now we give an algorithm to construct the maximal element in $B(\tW_{ad}, \pi(\mu), \pi(\s_{m, n}))$.

Let $\cs=\cup_{1 \le i \le j}\ZZ^{[i, j]}$, whose elements are called {\it segments}. Let $\eta \in \cs$ be a segment. Assume $\eta \in \ZZ^{[i, j]}$. We call $\h(\eta)=i$ and $\tt(\eta)=j$ the {\it head} and the {\it tail} of $\eta$ respectively. We call the positive integer $j-i+1$ the {\it size} of $\eta$. We set $$|\eta|=\sum_{k=\h(\eta)}^{\tt(\eta)} \eta(k), \qquad \ev(\eta)=\frac{1}{\tt(\eta)-\h(\eta)+1} |\eta|.$$ Let $[i', j'] \subseteq [i, j]$ be a sub-interval, we call the restriction $\eta|_{[i',j']}$ of $\eta$ to $[i', j']$ a {\it subsegment} of $\eta$ and write $\eta|_i=\eta|_{[i,i]}$. Let $\th$ be another segment such that $\h(\th)=\tt(\eta)+1$. We denote by $\eta \vee \th \in \ZZ^{[\h(\eta), \tt(\th)]}$ the natural concatenation of $\eta$ and $\th$. For $k \in \ZZ$, we denote by $\eta[k]$ the {\it $k$-shift} of $\eta$ defined by $\eta[k](i)=\eta(i+k)$. We say two segments are of the same {\it type} if they can be identified with each other up to some shift.

For $\eta \in \QQ^{[1,n]}$ we denote by $\C(\eta) \in \RR^2$ the convex hull of the points $(0,0)$ and $(k, |\eta|_{[1, k]}|)$ for $k \in [1,n]$. We say a subsegment $\g$ of $\eta$ is {\it sharp} if $$\ev(\g)=\max\{\ev(\g'); \g' \text{ is a subsegment of $\eta$ with } \h(\g')=\h(\g)\}$$ and $$\ev(\g)=\min\{\ev(\g'); \g' \text{ is a subsegment of $\eta$ with } \tt(\g')=\tt(\g)\}.$$ If $\eta=\g^1 \vee \g^2 \vee \cdots \vee \g^s$ with each $\g^k$ a sharp subsegment, then the points $(0,0)$ and $(\tt(\g^i), |\g^1 \vee \cdots \vee \g^i|)$ in $\RR^2$ for $i \in [1,s]$ lie on the boundary of $\C(\eta)$ and their convex hull is just $\C(\eta)$. We call the dominant vector $$\slp(\C(\eta))=(\A(\g^1) \vee \cdots \vee \A(\g^s)) \in \QQ^{[1,n]}$$ the {\it slope sequence} of $\C(\eta)$. Here for any $\g \in \cs$, we define $\A(\g) \in \QQ^{[\h(\g),\tt(\g)]}$ by $\A(\g)(i)=\ev(\g)$ for $i \in [\h(\g),\tt(\g)]$.

Let $\mu \in \ZZ^{[1,n]}$ be a dominant cocharacter. Set $\mu_{m,n}=\mu+\chi_{m,n}$. Then \begin{align*} \<\o_{i,n}, \pi(\mu_{m,n})\>&=\<\o_{i,n},\pi(\mu)\>-(\frac{m i}{n}-\lf \frac{m i}{n} \rf)\\ &=\<\o_{i,n}, \pi(\mu)+\pi(\varpi_{m,n})\>-\rc \<\o_{i,n}, \pi(\varpi_{m,n})\> \lc. \end{align*} According to the proof of Theorem \ref{main'} (1), the slope sequence $$\nu=\slp(\C(\pi(\mu_{m,n})))=\pi(\slp(\C(\mu_{m,n}))$$ is the unique maximal Newton point of $B(\tW_{ad}, \pi(\mu), \pi(\s_{m,n}))$.

\subsection{}\label{Eucliean} Similar to \cite[\S 5]{He10}, we use the Euclidean algorithm to give a recursive construction of $\chi_{m,n}$, which plays a crucial role in the proof of Proposition \ref{superbasic}.

Let $D=\{(m,n) \in \ZZ_{>0}^2; m<n \text{ are co-prime}\}$. We define $f: D \to D \sqcup \{(1,1), (0,1)\}$ by \begin{align*}f(m,n)=\begin{cases} (m(\lf\frac{n}{m}\rf+1)-n, m), &  \text{ if } \frac{n}{m} \geq 2;\\ (n-(n-m)\lf\frac{n}{n-m}\rf, n-m), & \text{ otherwise.} \end{cases}\end{align*}

Define two types of segments $1_{m,n}$ and $0_{m,n}$ by \begin{align*} 1_{m,n}&=\begin{cases} (0^{(\lf\frac{n}{m}\rf-1)}, 1), &  \text{ if } \frac{n}{m} \geq 2; \\ (0, 1^{(\lf\frac{n}{n-m}\rf)}), & \text{ otherwise,}\end{cases}\end{align*}
\begin{align*} 0_{m,n}&=\begin{cases} (0^{(\lf\frac{n}{m}\rf)},1), &  \text{ if } \frac{n}{m} \geq 2; \\ (0, 1^{(\lf\frac{n}{n-m}\rf-1)}), & \text{ otherwise,}\end{cases}\end{align*} where the superscript $^{(k)}$ means to repeat the entry $k$ times. Here we do not fix the head or tail of these segments yet, as we are going to apply various shifts to them below.

Set $\cs_1=\{\eta \in \cs; \eta(i) \in \{0, 1\} \text{ for } i \in [\h(\eta), \tt(\eta)]\}$. For $\eta \in \cs_1$ and $k \in [\h(\eta), \tt(\eta)]$, set \begin{align*}\eta(k)_{m,n}=\begin{cases} 1_{m,n}, &\text{ if } \eta(k)=1;\\ 0_{m,n}, & \text{ if } \eta(k)=0. \end{cases}\end{align*}

For $k \in [\h(\eta), \tt(\eta)]$, let $\eta_{m,n,k}$ be a shift of $\eta(k)_{m,n}$ whose head is determined recursively as follows: \begin{align*} \h(\eta_{m,n,k})=\begin{cases} \h(\eta), & \text{ if } k=\h(\eta);\\ \tt(\eta_{m,n,k-1})+1, & \text{ if } k > \h(\eta). \end{cases}\end{align*}
Now we define $\phi_{m,n}: \cs_1 \to \cs_1$ by $\phi_{m,n}(\eta)=\eta_{m,n,\h(\eta)} \vee \cdots \vee \eta_{m,n,\tt(\eta)}$ for $\eta \in \cs_1$.

If $f^{h-1}(m,n) \in D$, we set $\phi_{m,n,h}=\phi_{m,n} \circ \cdots \circ \phi_{f^{h-1}(m,n)}$. Using the Euclidean algorithm, one checks that $$\phi_{m,n,h}(\chi_{f^h(m,n)})=\chi_{m,n}.$$ We say a subsegment $\g$ of $\chi_{m,n}$ is of {\it level} $h$ if it is the image of some subsegment $\g^h$ of $\chi_{f^h(m,n)}$ under the map $\phi_{m,n,h}$. When $h=1$ and $\g^h$ is of size one, we say $\g$ is an {\it elementary} subsegment of $\chi_{m,n}$.

Let $\b^1$ and $\g^1$ be two segments of $\chi^1=\chi_{f(m,n)}$ and let $\g$ be a level one subsegment of $\chi=\chi_{m,n}$. Using the Euclidean algorithm, we have the following basic facts:

(a) $\ev(\b^1) \geq \ev(\g^1)$ if and only if $\ev(\phi_{m,n}(\b^1)) \geq \ev(\phi_{m,n}(\g^1))$.

(b) Each sharp subsegment of $\g$ with the same head is of level one.

(c) If, moreover, $\g$ is an elementary subsegment of $\chi$, then $\ua_\chi^{j} < \ua_\chi^{\h(\g)-1}$ and $\ua_\chi^{j} < \ua_\chi^{\tt(\g)}$ for $j \in [\h(\g), \tt(\g)-1]$.

(d) $\ua_{\chi^1}^i < \ua_{\chi^1}^j$  if and only if $\ua_\chi^{\tt(\phi_{m,n}(\chi^1 |_i))} < \ua_\chi^{\tt(\phi_{m,n}(\chi^1 |_j))}$.

(e) $\e_{m,n}(n)=1$.

\subsection{Proof of Proposition \ref{superbasic}}\label{proof-5.1}
For a sequence of (distinct) elements $i_1, i_2, \dots, i_r$ in $[1, n]$, we denote by $\cyc(i_1, i_2, \dots, i_r) \in \fs_n$ the cyclic permutation $i_1 \mapsto i_2 \mapsto \cdots \mapsto i_r \mapsto i_1$, which acts trivially on the remaining elements of $[1,n]$.

For $\eta \in \cs$ we set $$x_\eta=\cyc(\h(\eta), \h(\eta)+1, \dots, \tt(\eta)) \in \fs=\cup_{i=1}^\infty \fs_i.$$ Similarly, for a sequence $\bc=(c^1, \dots, c^s)$ of segments, we set $x_{\bc}=x_{c^1, \dots, c^s}=x_{c^1} \cdots x_{c^s}$. If $\eta=c^1 \vee \cdots \vee c^s$, we say $\bc$ is a {\it decomposition of $\eta$}. Now we are ready to prove Proposition \ref{superbasic}.

Write $\chi=\chi_{m,n}$, $\th=\mu_{m,n}=\mu+\chi$ and $\e=\e_{m,n}$. For $h \in \ZZ_{>0}$ we set $\phi_h=\phi_{m,n,h}$ and $\chi^h=\chi_{f^h(m,n)}$. By \S\ref{polygon}, we have $\nu=\pi(\slp(\C(\th)))$. The proof will proceed as follows. First we construct a suitable sharp decomposition $\bc$ of $\th$. One checks directly $\pi(\nu_{w_{\bc}, \id})=\pi(\slp(\C(\th)))=\nu$, where $w_{\bc}=t^\th x_{\bc} \in \tW$. Then we show that $$\e w_{\bc} \e\i \le \e t^\th x_\th \e\i= t^{\e(\mu)} \s_{m,n},$$ where the last equality follows from Lemma \ref{equa} below. Set $\tw=\e w_{\bc} \e \i \s_{m, n} \i$. Then $\tw \le t^{\e(\mu)}$ and $\pi(\nu_{\tw, \s_{m, n}})=\pi(\nu_{\e w_{\bc} \e \i, id})=\e(\nu)$. This completes the proof of Proposition \ref{superbasic}.%Here the last equality follows from \cite[Lemma 5.1]{He10}.

\begin{lem} \label{equa}
We have $\s_{m,n}=\e t^\chi x_\chi \e\i$.
\end{lem}
\begin{proof}
We use the explicit descriptions of $\geq_\chi$ and $\e$ in \S \ref{5.3} via the identification of $[1, n]$ with $\ZZ / n\ZZ$ in the natural way. Then ${u_{1,n}}=x_\chi$ is the permutation $i \mapsto i+1$ on $\ZZ / n\ZZ$. Since $\s_{m,n}=\s_{1,n}^m$, we have $u_{m,n}(i)=i+m$ for $i \in \ZZ / n\ZZ$. On the other hand, we already know that $\e(i)=mi+1$ for $i \in \ZZ / n\ZZ$. Therefore, $\e x_\chi \e\i$ is the permutation $mi+1 \mapsto m(i+1)+1$ on $\ZZ / n\ZZ$, which equals $u_{m,n}$ as desired. It remains to show $\e(\chi)=\varpi_{m,n}$. For $1 \leq i < j \leq n$ we have $\e(\chi)(i)=\chi(m\i \cdot (i-1)) \geq \chi(m\i \cdot (j-1))=\e(\chi)(j)$ since $m\i \cdot (i-1) >_{\chi} m\i \cdot (j-1)$. Thus $\e(\chi)$ is dominant and equals $\varpi_{m,n}$. The proof is finished.
\end{proof}

\

Assume $I(\mu)=\{j \in [1,n-1]; \<\a_j, \mu\> \neq 0\}=\{b_1, b_2, \dots, b_{r-1}\}$ with $b_1 < b_2 < \cdots < b_{r-1}$. We set $b_0=0$ and $b_r=n$. Set $\th^i=\th|_{[b_{i-1}+1,b_i]}$ for $i \in [1,r]$. Then $\th=\th^1 \vee \cdots \vee \th^r$. Suppose we have a sharp decomposition $\bc_i$ of $\th^i$ for $i \in [1, r]$. Since $\chi \in \{0,1\}^{[1,n]}$ and $\th=\mu+\chi$, for any subsegment $\eta^i$ (resp. $\eta^j$) of $\th^i$ (resp. $\th^j$) we have $\ev(\eta^i) \geq \ev(\eta^j)$ if $i<j$. Therefore the natural union $\bc=\bc_1 \vee \cdots \vee \bc_r$ forms a sharp decomposition of $\th$.

\

Let $1 \leq i \leq r$. We will construct inductively the subsegments $\z_i^j, \g_i^j, \xi_i^j$ for $j \in [1, l_i]$ (some of them might be empty) such that

(a) $\g^0_i=\chi |_{[b_{i-1}+1, b_j]}$ and $\g_i^{j-1}=\z_i^j \vee \g_i^j \vee \xi_i^j$ for $j \in [1, l_i]$;

(b) $\z_i^j$ and $\xi_i^j$ are sharp subsegments of $\g_i^{j-1}$; any sharp subsegment of $\g_i^j$ is also a sharp subsegment of $\g_i^{j-1}$; $\g_i^{l_i}$ is a sharp subsegment of itself ({\it self-sharp});

(c) For any $j$, $\e z_{i,j-1} \e\i \ge \e z_{i,j} \e\i$.

Here \begin{align*} z_{i,j}&=t^\th y_{i-1} x_i^j v_{i,j}; \\ y_i &=x_{\bc_1} \cdots x_{\bc_{i}}; \\ x_i^j &=x_{\z_i^1, \dots, \z_i^j, \xi_i^j, \dots, \xi_i^1}; \\ v_{i,j}&=x_{\g_i^j} x_{\th^{i+1} \vee \cdots \vee \th^r} \cyc(\tt(\g_i^j), n)) \\ &=\cyc(\h(\g_i^j), \dots, \tt(\g_i^j), b_i+1, \dots, n).\end{align*} Assume we have (a), (b) and (c) for all $i$ and $j$. Set ${\z'}_i^j=\th |_{[\h(\z_i^j), \tt(\z_i^j)]}$, ${\g'}_i^{l_i}=\th |_{[\h(\g_i^{l_i}), \tt(\g_i^{l_i})]}$ and ${\xi'}_i^j=\th |_{[\h(\xi_i^j), \tt(\xi_i^j)]}$. Then $$\bc_i=({\z'}_i^1, {\z'}_i^2, \dots, {\z'}_i^{l_i}, {\g'}_i^{l_i}, {\xi'}_i^{l_i}, \dots, {\xi'}_i^2, {\xi'}_i^1)$$ forms a sharp decomposition of $\th^i$, and $$\e t^\th x_\th \e\i = \e z_{1,0} \e\i \ge \cdots \ge \e z_{1, l_1+1}\e\i=\e z_{2, 0} \e\i \ge \cdots \ge \e z_{r, l_r+1} \e\i =\e w_{\bc} \e\i$$ as desired.

\

The construction is as follows. Suppose for $1 \leq k < i$ and $0 \leq l \leq j$, $\bc_k$, $z_i^l$, $\xi_i^l$, $\g_i^l$ are already constructed, and moreover $\e z_{i,j-1} \e\i \ge \e z_{i,j} \e\i$. We construct $\z_i^{j+1}, \g_i^{j+1}, \xi_i^{j+1}$ and show that $\e z_{i,j} \e\i \ge \e z_{i, j+1} \e\i$.

If $\g_i^j$ is empty, there is nothing to do. Otherwise, we assume $\g_i^j$ is of level $h$ but not of level $h+1$. Then $\g_i^j=\phi_h(\io)$ for some subsegment $\io$ of $\chi^h$.

Case (I): {\it $\io$ is not a subsegment of any elementary subsegment of $\chi^h$.} Then there exist unique subsegments $\z$, $\g$ and $\xi$ of $\chi^h$ such that $\g$ is of level one, $\z$ (resp. $\xi$) is a proper subsegment of some elementary segment of $\chi^h$ with the same tail (resp. head), and $\io=\z \vee \g \vee \xi$. Notice that at least two of $\z$, $\g$ and $\xi$ are nonempty.

Define $\z_i^{j+1}=\phi_h(\z)$, $\g_i^{j+1}=\phi_h(\g)$ and $\xi_i^{j+1}=\phi_h(\xi)$. Note that $\ev(\chi^h|_{[\h(\z), \tt(\z)]})$ is maximal among all subsegments of $\chi^h$ with the same head and $\ev(\chi^h|_{[\h(\xi), \tt(\xi)]})$ is minimal among all subsegments of $\chi^h$ with the same tail. Therefore, (b) follows from \S\ref{Eucliean} (a) $\&$ (b). To prove (c), it suffices to show that
\begin{gather*}
\tag{d} \e z_{i, j+1} \e\i \le \e\ z_{i,j}\ \cyc(n, \tt(\z_i^{j+1}))\ \e\i; \\
\tag{e} \e\ z_{i,j}\ \cyc(n, \tt(\z_i^{j+1}))\ \e\i \le \e z_{i, j} \e \i.
\end{gather*}

Note that \begin{align*} \e z_{i,j+1} \e\i= \begin{cases} \e\ z_{i,j}\ \cyc(n, \tt(\z_i^{j+1}))\ \cyc(\h(\xi_i^{j+1})-1, \tt(\xi_i^{j+1}))\ \e\i, & \text{ if }\g \neq \emptyset; \\ \e\ z_{i,j}\ \cyc(n, \tt(\z_i^{j+1}))\ \cyc(n, \tt(\xi_i^{j+1}))\ \e\i, & \text{ otherwise.} \end{cases} \end{align*} Here we take $\cyc(n, \tt(\z_i^{j+1}))$ (resp. $\cyc(\h(\xi_i^{j+1})-1, \tt(\xi_i^{j+1}))$ and $\cyc(n, \tt(\xi_i^{j+1}))$) to be the identity element of $\fs_n$ if $\z$ (resp. $\xi$) is empty, in which case the inequality (e) (resp. (d)) becomes a priori an equality.

Now we prove (d). We suppose $\xi$ is nonempty. Otherwise, there is nothing to prove. First we assume $\g \neq \emptyset$. By \S\ref{Eucliean} (c), we have $\ua_{\chi^h}^{\h(\xi)-1} > \ua_{\chi^h}^{\tt(\xi)}$. Hence by \S\ref{Eucliean} (d), $\e(\h(\xi_i^{j+1})-1) < \e(\tt(\xi_i^{j+1}))$ and $\a=\e(e_{\h(\xi_i^{j+1})-1} - e_{\tt(\xi_i^{j+1})})$ is a positive root. Then (d) is equivalent to the following inequality \begin{align*}\< \a, \e z_{i,j+1}\i \e \i (\ba)\> & =\<\a, \e (v_{i,j} x_i^j y_{i-1})\i \e\i (\ba -\e(\th))\>  \\ &= \<\e (v_{i,j} x_i^j y_{i-1}) \e\i(\a), \ba-\e(\th)\> \\ & = -\<\e(e_{b_i+1}-e_{\h(\xi_i^{j+1})}), \e(\th)\> + \<\e(e_{b_i+1}-e_{\h(\xi_i^{j+1})}), \ba\> \\ &= \th(\h(\xi_i^{j+1}))-\th(b_i+1) + \<\e(e_{b_i+1}-e_{\h(\xi_i^{j+1})}), \ba\> > 0,\end{align*} where $\ba$ is the base alcove defined in \S \ref{alcove}. Note that $1 > \<\e(e_{b_i+1}-e_{\h(\xi_i^{j+1})}), \ba\> > -1$. Therefore, to prove (d), we have to show either $\th(\h(\xi_i^{j+1})) > \th(b_i+1)$ or $\th(\h(\xi_i^{j+1})) = \th(b_i+1)$ and $\e(b_i+1) < \e(\h(\xi_i^{j+1}))$. Note that we always have $\th(\h(\xi_i^{j+1})) \geq \th(b_i+1)$. If $\th(\h(\xi_i^{j+1})) = \th(b_i+1)$, then $\chi(b_i+1)=1>0=\chi(\h(\xi_i^{j+1}))$. Hence $\e(b_i+1) < \e(\h(\xi_i^{j+1}))$ as desired. Now we assume $\g= \emptyset$. Note that $\e(n) < \e(\tt(\xi_i^{j+1}))$. Then (d) follows by a similar argument as in the case of $\g \neq \emptyset$ with $\h(\xi_i^{j+1})-1$ replaced by $n$.

To prove (e), again we suppose that $\z$ is nonempty. Since one of $\g$ and $\xi$ is nonempty, $\tt(\z_i^{j+1}) \neq n$. By \S\ref{Eucliean} (e), $1=\e(n) < \e(\tt(\z_i^{j+1}))$. As in the proof of (d), we see that (e) holds if $\th(\tt(\z_i^{j+1})+1) < \th(\h(\z_i^{j+1}))$. Otherwise, we have $\th(\tt(\z_i^{j+1})+1) = \th(\h(\z_i^{j+1}))$, which implies that $\chi(\h(\z_i^{j+1}))=\chi(\tt(\z_i^{j+1})+1)=0$. Since $\z$ is a proper subsegment of some elementary segment of $\chi^h$ and shares the same tail with it, by \S\ref{Eucliean} (c) we have that $\ua_{\chi^h}^{\tt(\z)} > \ua_{\chi^h}^{\h(\z)-1}$. Hence by \S\ref{Eucliean} (d) and that $\chi(\tt(\z_i^{j+1})+1)=\chi(\h(\z_i^{j+1}))=0$, we have $\ua_\chi^{\tt(\z_i^{j+1})+1} > \ua_\chi^{\h(\z_i^{j+1})}$. So $\e(\h(\z_i^{j+1})) > \e(\tt(\z_i^{j+1})+1)$ and (e) holds.

\

Case (II): {\it $\io$ is a subsegment of some elementary subsegment of $\chi^h$.} We define $l_i=j$ and the construction of $\bc_i$ is finished. One checks directly that $\io$ is self-sharp, hence so is $\g_i^{l_i}=\phi_h(\io)$ by \S\ref{Eucliean} (a) $\&$ (b). If $\tt(\g_i^{l_i})=n$, the induction step is finished. Otherwise, it remains to show \begin{align*} \tag{f} \e z_{i,l_i} \e\i \ge \e\ z_{i,l_i}\ \cyc(\tt(\g_i^{l_i}), n)\ \e\i=\e z_{i,l_i+1} \e\i. \end{align*} Note that $1=\e(n) < \e(\tt(\g_i^{l_i}))$. Again, we see that (f) holds if $\th(\h(\g_i^{l_i})) > \th(b_i+1)$. Otherwise, we have $\th(\h(\g_i^{l_i})) = \th(b_i+1)$, $\chi(\h(\g_i^{l_i}))=0$ and $\chi(b_i+1)=1$ since $b_i \in I(\mu)$. Hence $\e(\h(\g_i^{l_i})) > \e(b_i+1)$ and (f) still holds.

\begin{example} Finally we provide an example.

Let $n=8$, $m=5$ and $\mu=(1, 1, 1, 0, 0, 0, 0, 0)$. Then $\chi_{m,n}=(0, 1, 0, 1, 1, 0, 1, 1) \in \ZZ^8$, $\e_{m,n}=\cyc(1, 6, 7, 4, 5, 2, 3, 8)$ and $u_{m, n}=\cyc(6, 3, 8, 5, 2, 7, 4, 1)$.

Note that $I(\mu)=3$. Applying the algorithm in the proof of Proposition \ref{superbasic}, we obtain the following sharp decomposition: $$\mu_{m,n}=(1, 2, 1, 1, 1, 0, 1, 1)=(1, 2) \vee(1) \vee (1, 1) \vee (0, 1, 1).$$ Hence $\nu=(\frac{1}{2}, \frac{1}{2}, 0, 0, 0, -\frac{1}{3}, -\frac{1}{3}, -\frac{1}{3})$. Moreover, one checks that \begin{align*} t^{\e_{m,n}(\mu)} \s_{m,n} & \ge t^{\e_{m,n}(\mu)} \s_{m,n} \cyc(8, 3)\\ & \ge t^{\e_{m,n}(\mu)} \s_{m,n} \cyc(8, 3) \cyc(1, 3)\\ & \ge t^{\e_{m,n}(\mu)} \s_{m,n} \cyc(8, 3) \cyc(1, 3) \cyc(1, 2).\end{align*} Set $\tw=t^{\e_{m,n}(\mu)} \s_{m,n} \cyc(8, 3) \cyc(1, 3) \cyc(1, 2) \s_{m, n} \i$. Then $\tw \le t^{\e_{m, n}(\mu)}$ and $\pi(\bar \nu_{\tw, \s_{m,n}})=\nu$. This verifies Proposition \ref{superbasic} in this case.
\end{example}

\appendix

\section{The set $B(G, \{\mu\})$}

In the appendix, we discuss the relation between the set $B(\tW, \mu, \s)$ defined in \S\ref{1.4} and the set $B(G, \mu)$ for $p$-adic groups.

\subsection{} Recall that $F$ is a finite field extension of $\QQ_p$, $L$ is the completion of the maximal unramified extension of $F$ and $G$ is a connected reductive algebraic group over $F$. We first discuss the Iwahori-Weyl group of $G$ over $L$. We follow \cite{HR}.

Let $S$ be a maximal $L$-split torus that is defined over $F$ and let $T$ be its centralizer. Since $G$ is quasi-split over $L$, $T$ is a maximal torus.

Let $N$ be the normalizer of $T$. The {\it finite Weyl group} associated to $S$ is $W_0=N(L)/T(L).$ The {\it Iwahori-Weyl group} associated to $S$ is $\tW=N(L)/T(L)_1$, where $T(L)_1$ denotes the unique Iwahori subgroup of $T(L)$.

Let $\G=\Gal(\bar L/L)$. As in \cite[p.\,195]{HR} and \cite[\S 4.2]{PRS}, one associates a reduced root system $R$ to $(G, T)$. Let $V=X_*(T)_\G \otimes \RR=X_*(S) \otimes \RR$. We fix a $\s$-invariant alcove $\aa_G$ in the apartment of $S$ along with a special vertex of $\aa_G$. The special vertex allows us to identify $V$ with the apartment of $S$. The alcove $\aa_G$ is contained in a unique (relative) Weyl chamber of $V$, which we call the dominant chamber. The hyperplanes in $V$ then give a reduced root system $R$. We have the semi-direct product $$\tW=X_*(T)_\G \rtimes W_0.$$

The group $X_*(T)_\G$ is not torsion-free in general. However, by \cite[\S8.1]{HH}, the torsion part $X_*(T)_{\G, tor}$ lies in the center of $\tW$ and one may identify the extended affine Weyl group of $\mathfrak R$ with $\tW/X_*(T)_{\G, tor}$. For our purpose, it suffices to consider the case where $X_*(T)_\G$ is torsion-free. In this case, the root system $R$, together with the cocharacter group $X_*(T)_\G$, defines a reduced datum $\mathfrak R$, and $\tW$ is the extended afffine Weyl group of $\mathfrak R$ introduced in section \ref{1}.

The Iwahori-Weyl group $\tW$ contains the affine Weyl group $W_a$ as a normal subgroup and $$\tW=W_a \rtimes \Omega,$$ where $\Omega \cong \pi_1(G)_{\G}$ is the normalizer of the alcove $\aa_G$. The Bruhat order on $W_a$ extends in a natural way to $\tW$. The Frobenius morphism $\s$ induces an action on $\tW$, which we denote by $\Ad(\s)$.

%Let $G_{sc}$ be the simply connected cover of the derived group of $G$. Denote by $T_{sc}$ the maximal torus of $G_{sc}$ given by the choice of $T$. Then we have a natural injective map $X_*(T_{sc})_\G \to X_*(T)_\G$. We fix a special vertex of $\aa_G$ and represent $\tW$ and $W_a$ as \begin{gather*} \tW=X_*(T)_\G \rtimes W_0=\{t^\l w; \l \in X_*(T)_\G, w \in W_0\}, \\ W_a=X_*(T_{sc})_\G \rtimes W_0=\{t^\l w; \l \in X_*(T_{sc})_\G, w \in W_0\}. \end{gather*}

\subsection{} Recall that $\{\mu\}$ is a geometric conjugacy class of cocharacters of $G$. We may regard $\{\mu\}$ as a conjugacy class in $X_*(T)$ under the absolute Weyl group. Following \cite[\S4.3]{PRS}, let $\tilde \L_{\{\mu\}} \subseteq \{\mu\}$ be the subset of cocharacters which are $B$-dominant for some Borel subgroup $B$ defined over $L$ with $B \supseteq T$. Then $\tilde \L_{\{\mu\}}$ is a single $W_0$-orbit. Let $\L_{\{\mu\}}$ be the image of $\tilde \L_{\{\mu\}}$ in $X_*(T)_\G$. The $\{\mu\}$-admissible set is defined by $$\Adm(\{\mu\})=\{w \in \tW; w \le t^\xi \text{ for some } \xi \in \L_{\{\mu\}}\}.$$

Let $\tilde \mu$ be the unique dominant cocharacter in $\tilde \L_{\{\mu\}}$ and $\mu$ be its image in $\L_{\{\mu\}}$. Then $\L_{\{\mu\}}=W_0 \cdot \mu$ and $\Adm(\{\mu\})$ equals $\Adm(\mu)$ defined in \S\ref{1.4}.\footnote{In the function field case, it is shown in \cite[Remark 2.11]{Ri} that $\Adm(\{\mu\})=\{w \in \tW; w \le t^\xi \text{ for some } \xi \in im\{\mu\} \subseteq X_*(T)_\G\}$. We do not need this result here.}

\subsection{}
The set $B(G)$ of $\s$-conjugacy classes of $G(L)$ is classified by Kottwitz in \cite{Ko1} and \cite{Ko2}.

For any $b \in G(L)$, we denote by $[b]$ the $\s$-conjugacy class of $G(L)$ that contains $b$. Let $\G_F=\Gal(\bar L/F)$ be the absolute Galois group of $F$. Let $\k_G: B(G) \to \pi_1(G)_{\G_F}$ be the Kottwitz map \cite[\S 7]{Ko2}. This gives one invariant.

Another invariant is given by the Newton map. To an element $b \in G(L)$, we associate its Newton point $\bar \nu_b \in X_*(T)_{\G} \otimes \RR$.

By \cite[\S 4.13]{Ko2}, the map $$B(G) \to \pi_1(G)_{\G_F} \times (X_*(T)_{\G} \otimes \RR), \qquad b \mapsto (\k_G(b), \bar \nu_b)$$ is injective.

For any $w \in \tW$, we choose a representative in $N(L)$ and also write it as $w$. The map $N(L) \to G(L)$ induces a map $\tW \to B(G)$. By \cite[\S 3]{He99} and \cite[2.4]{GHN}\footnote{\cite[2.4]{GHN} is stated for adjoint groups, but the result for arbitrary groups holds by combining with the reduction argument in \cite[2.3]{GHN}.}, this map is surjective. The restrictions of the Kottwitz map and the Newton map on $\tW \subseteq G(L)$ are the maps $\k_{\tW, \s}$ and $w \mapsto \bar \nu_{w, \s}$ defined in \S\ref{1.3}.

\subsection{} In this section, we assume furthermore that $G$ is a quasi-split connected reductive group over $F$ and that $H$ is an inner form of $G$. We denote by $\s_G$ and $\s_H$ the Frobenius morphisms of $G$ and $H$ respectively. Via the canonical isomorphism $X_*(T)_{\G} \otimes \RR \cong (X_*(T) \otimes \RR)^\G$, we may identify $\mu^\diamondsuit_{\s_G}$ in \S\ref{1.4} with $[\G_F: \G_{F, \tilde \mu}] \i \sum_{\t \in \G_F/\G_{F, \tilde \mu}} \t(\tilde \mu)$ in \cite[(6.1.1)]{Ko2}, where $\G_{F, \tilde \mu}$ is the isotropy group of $\mu$ in $\G_F$. By Lemma \ref{suites}, $\mu^\diamondsuit_{\s_G}=\mu^{\diamondsuit}_{\s_H}$.

Let $\mu^\sharp$ be the image of $\tilde \mu$ under the natural map $X_*(T) \to \pi_1(H)_{\G_F}$. %Following \cite[\S 6]{Ko2}, we may regard $\tilde \mu$ as a character of $Z(\hat{G})^{\G_F}=Z(\hat{H})^{\G_F}$, where $Z(\hat{G})$ is the center of the Langlands dual group $\hat{G}$ of $G$ and $Z(\hat{H})$ is the center of the Langlands dual group $\hat{H}$ of $H$. Let $\mu^\sharp$ be the image of this character via the natural isomorphism $X^*(Z(\hat{H})^{\G_F}) \cong \pi_1(H)_{\G_F}$.

Set $$B(H, \{\mu\})=\{[b] \in B(H); \k_H(b)=\mu^\sharp, \bar \nu_b \leq \mu^\diamondsuit_{\s_H}\}.$$ Then we may identify $B(H, \{\mu\})$ with $B(\tW, \mu, \s_H)$. Theorem \ref{main'} may be reformulated as follows:

The set $B(H, \{\mu\})$ contains a unique maximal element
and this element is represented by an element in $\Adm(\{\mu\})$.

\section*{Acknowledgement} We thank the referees for their careful reading of this paper and many useful comments.

\end{document}